\documentclass[twocolumn,amsmath,amssymb,aps,secnumarabic,%
    nofootinbib,superscriptaddress]{revtex4-1}
\usepackage{mathptmx,amsthm,letltxmacro}

\renewcommand{\thesection}{\arabic{section}}
\renewcommand{\figurename}{Figure}

\renewcommand{\arraystretch}{2}
\newcommand{\adj}{\operatorname{adj}}
\newcommand{\Ball}{\operatorname{Ball}}
\newcommand{\Candle}{\operatorname{Candle}}
\newcommand{\CH}{\mathbb{C}\mathrm{H}}
\newcommand{\dd}{\mathrm{d}}
\newcommand{\hA}{\hat{A}}
\newcommand{\HH}{\mathbb{H}\mathrm{H}}
\newcommand{\htop}{\operatorname{h_{\mathrm{top}}}}
\newcommand{\hvol}{\operatorname{h_{\mathrm{vol}}}}
\newcommand{\hmu}{\operatorname{h_\mu}}
\newcommand{\hY}{\hat{Y}}
\newcommand{\inj}{\operatorname{inj}}
\newcommand{\LCD}{\operatorname{LCD}}
\newcommand{\OH}{\mathbb{O}\mathrm{H}}
\newcommand{\RH}{\R\mathrm{H}}
\newcommand{\Ric}{\operatorname{Ric}}
\newcommand{\R}{\mathbb{R}}
\newcommand{\RRic}{\operatorname{\sqrt{R}ic}}
\newcommand{\tM}{\tilde{M}}
\newcommand{\Tr}{\operatorname{Tr}}
\newcommand{\Vol}{\operatorname{Vol}}

\newcommand{\defeq}{\stackrel{\mathrm{def}}{=}}
\newcommand{\ie}{\i.e.}
\newcommand{\braket}[1]{\langle {#1} \rangle}

\newtheorem{theorem}{Theorem}[section]
\newtheorem{proposition}[theorem]{Proposition}
\theoremstyle{remark}
\newtheorem*{remark}{Remark}

\renewcommand{\sec}[1]{Section~\ref{#1}}
\newcommand{\thm}[1]{Theorem~\ref{#1}}
\newcommand{\prop}[1]{Proposition~\ref{#1}}
\newcommand{\eq}[2]{\begin{equation}\label{#1}#2\end{equation}}

\begin{document}

\title{A refinement of G\"unther's candle inequality}
\author{Beno\^{\i}t R. Kloeckner}
\email{benoit.kloeckner@ujf-grenoble.fr}
\affiliation{Institut Fourier, Universit\'e de Grenoble I}
\author{Greg Kuperberg}
\email{greg@math.ucdavis.edu}
\thanks{Supported by NSF grant CCF \#1013079.}
\affiliation{Department of Mathematics, University of California, Davis}

\begin{abstract}
\centerline{\textit{\normalsize Dedicated to our friends Sylvain Gallot and 
    Albert Schwarz}}
\vspace{\baselineskip}
We analyze an upper bound on the curvature of a Riemannian manifold, using
``$\RRic$" curvature, which is in between a sectional curvature bound and a
Ricci curvature bound.  (A special case of $\RRic$ curvature was previously
discovered by Osserman and Sarnak for a different but related purpose.)
We prove that our $\RRic$ bound implies G\"unther's inequality on the
candle function of a manifold, thus bringing that inequality closer in
form to the complementary inequality due to Bishop.
\end{abstract}

\maketitle

\section{Introduction}

Two important relations between curvature and volume in differential geometry
are Bishop's inequality \cite[\S11.10]{BC:book}, which is an upper bound on
the volume of a ball from a lower bound on Ricci curvature, and G\"unther's
inequality \cite{Gunther:volume}, which is a lower bound on volume from
an upper bound on \emph{sectional} curvature.  Bishop's inequality has
a weaker hypothesis then G\"unther's inequality and can be interpreted
as a stronger result.  The asymmetry between these inequalities is a
counterintuitive fact of Riemannian geometry.

In this article, we will partially remedy this asymmetry.  We will define
another curvature statistic, the root-Ricci function, denoted $\RRic$, and
we will establish a comparison theorem that is stronger than G\"unther's
inequality\footnote{We take the ``ic" in the Ricci tensor $\Ric$ to mean
taking a partial trace of the Riemann tensor $R$, but we take a square
root first.}.  $\RRic$ is not a tensor because it involves square roots
of sectional curvatures, but it shares other properties with Ricci curvature.

After the first version of this article was written, we learned that a
special case of $\RRic$ was previously defined by Osserman and Sarnak
\cite{OS:entropy}, for the different but related purpose of estimating
the entropy of geodesic flow on a closed manifold. (See \sec{s:exp}.)
Although their specific results are different, there is a common motivation
arising from volume growth in a symmetric space.

\subsection{Growth of the complex hyperbolic plane}
\label{s:complex}

Consider the geometry of the complex hyperbolic plane $\CH^2$.  In this
4-manifold, the volume of a ball of radius $r$ is
$$\Vol(B(r)) = \frac{\pi^2}{2} \sinh(r)^4 \sim \frac{\pi^2}{32}\exp(4r).$$
The corresponding sphere surface volume has a factor of $\sinh(2r)$ from
the unique complex line containing a given geodesic $\gamma$, which has
curvature $-4$, and two factors of $\sinh(r)$ from the totally real planes
that contain $\gamma$, which have curvature $-1$.  G\"unther's inequality
and Bishop's inequality yield the estimates
$$\frac{\pi^2}{48}\exp(3\sqrt{2} r) \gtrsim
    \Vol(B(r)) \gtrsim \frac{\pi^2}{12}\exp(3r).$$
The true volume growth of balls in $\CH^2$ (and in some other cases,
see \sec{s:exp}) is governed by the average of the square roots of the
negatives of the sectional curvatures.  This is how we define the $\RRic$
function, for each tangent direction $u$ at each point $p$ in $M$.

\subsection{Root-Ricci curvature}

Let $M$ be a Riemannian $n$-manifold with sectional curvature $K \le
\rho$ for some constant $\rho \ge 0$; we will implicitly assume that
$\rho\ge\kappa$.  For any unit tangent vector $u \in UT_pM$ with $p \in M$,
we define
$$\RRic(\rho,u) \defeq \Tr(\sqrt{\rho-R(\cdot,u,\cdot,u)}).$$
Here $R(u,v,w,x)$ is the Riemann curvature tensor expressed as a tetralinear
form, and the square root is the positive square root of a positive
semidefinite matrix or operator. 

The formula for $\RRic(\rho,u)$ might seem arcane at first glance.
Regardless of its precise form, the formula is both local (\ie, a function
of the Riemannian curvature) and also optimal in certain regimes.  Any such
formula is potentially interesting.  One important, simpler case is
$\rho=0$, which applies only to non-positively curved manifolds:
$$\RRic(0,u) = \Tr(\sqrt{-R(\cdot,u,\cdot,u)}).$$
In other words, $\RRic(0,u)$ is the sum of the square roots of the
sectional curvatures $-K(u,e_i)$, where $(e_i)$ is a basis of $u^\perp$
that diagonalizes the Riemann curvature tensor.  This special case was
defined previously by Osserman and Sarnak \cite{OS:entropy} (\sec{s:exp}),
which in their notation would be written $-\sigma(u)$.

For example, when $M=\CH^2$, one sectional curvature $K(u,e_i)$ is $-4$
and the other two are $-1$, so
$$\RRic(0,u)= \sqrt{4}+\sqrt{1}+\sqrt{1} = 4,$$
which matches the asymptotics in \sec{s:complex}.

In the general formula $\RRic(\rho,u)$, the parameter $\rho$ is important
because it yields sharper bounds at shorter length scales.  In particular,
in the limit $\rho\to\infty$, $\RRic(\rho,u)$ becomes equivalent to Ricci
curvature.  \sec{s:relations} discusses other ways in which $\RRic$ fits
the framework of classical Riemannian geometry.  Our definition for general
$\rho$ was motivated by our proof of the refined G\"unther inequality, more
precisely by equation \eqref{e:A}.  The energy \eqref{e:energy} of a curve
in a manifold can be viewed as linear in the curvature $R(\cdot,u,\cdot,u)$.
We make a quadratic change of variables to another matrix $A$, to express
the optimization problem as quadratic minimization with linear constraints;
and we noticed an allowable extra parameter $\rho$ in the quadratic change
of variables.

Another way to look at root-Ricci curvature is that it is equivalent to
an average curvature, like the normalized Ricci curvature $\Ric/(n-1)$,
but after a reparameterization.  By analogy, the $L^p$ norm of a function,
or the root-mean-square concept in statistics, is also an average of
quantities that are modified by the function $f(x) = x^p$.  In our case,
we can obtain a type of average curvature which is equivalent to $\RRic$
if we conjugate by $f(x) = \sqrt{\rho - x}$.  Taking this viewpoint,
we say that the manifold $M$ is of \emph{$\RRic$ class $(\rho,\kappa)$}
if $K \le \rho$, and if also
$$\frac{\RRic(\rho,u)}{n-1} \ge \sqrt{\rho-\kappa}$$
for all $u \in UTM$.  This is the $\RRic$ curvature analogue of the
sectional curvature condition $K \le \kappa$.

\subsection{A general candle inequality}

The best version of either G\"unther's or Bishop's inequality is not directly
a bound on the volume of balls in $M$, but rather a bound on the logarithmic
derivative of the \emph{candle function} of $M$.  Let $\gamma=\gamma_u$ be a
geodesic curve in $M$ that begins at $p = \gamma(0)$ with initial velocity
$u\in UT_pM$.  Then the candle function $s(\gamma,r)$ is by definition the
Jacobian of the map $u\mapsto \gamma_u(r)$. In other words, it is defined by
the equations
$$\dd q = s(\gamma_u,r)\, \dd u \,\dd r \qquad q = \gamma_u(r) = \exp_p(ru),$$
where $\dd q$ is Riemannian measure on $M$, $\dd r$ is Lebesgue measure
on $\R$, and $\dd u$ is Riemannian measure on the sphere $UT_pM$.
This terminology has the physical interpretation that if an observer is
at the point $q$ in $M$, and if a unit candle is at the point $p$, then
$1/s(\gamma,r)$ is its apparent brightness\footnote{Certain distant objects
in astronomy with known luminosity are called \emph{standard candles}
and are used to estimate astronomical distances.}.

The candle function $s_\kappa(r)$ of a geometry of constant curvature
$\kappa$ is given by
$$s_{\kappa}(r) = \begin{cases} \left(\frac{\sin(\sqrt{\kappa}r)}
    {\sqrt{\kappa}}\right)^{n-1} & \kappa > 0 \\
    r^{n-1} & \kappa = 0 \\ \left(\frac{\sinh(\sqrt{-\kappa}r)}
    {\sqrt{-\kappa}}\right)^{n-1} & \kappa < 0
    \end{cases}.$$

\begin{theorem} Let $M$ be a Riemannian $n$-manifold is of $\RRic$ class
$(\rho,\kappa)$ for some $\kappa \le \rho \ge 0$.  Then
$$(\log s(\gamma,r))' \ge (\log s_{\kappa}(r))'$$
for every geodesic $\gamma$ in $M$, when $2r\sqrt{\rho} \le \pi$.
\label{th:main} \end{theorem}
The prime denotes the derivative with respect to $r$.

When $\rho = 0$, the conclusion of \thm{th:main} is identical to G\"unther's
inequality for manifolds with $K \le \kappa$, but the hypothesis is strictly
weaker.  When $\rho > 0$, the curvature hypothesis is weaker still, but the
length restriction is stronger.  The usual version of the inequality holds
up to a distance of $\pi/\sqrt{\kappa}$.  For our distance restriction,
we replace $\kappa$ with $\rho$ and divide by 2.

The rest of this article is organized as follows. In \sec{s:relations}
we give several relations between curvature bounds and volume comparisons.
In \sec{s:applications} we list applications of \thm{th:main},
and we prove \thm{th:main} in \sec{s:proof}.

\acknowledgments

The authors would like to thank Sylvain Gallot and John Hunter for especially
helpful conversations.

\section{Relations between conditions}
\label{s:relations}

\subsection{Candle conditions}

We first mention two interesting properties of the candle function
$s(\gamma,r)$:
\begin{enumerate}
\item[\textbf{1.}] $s(\gamma,r)$ vanishes when $\gamma(0)$ and $\gamma(r)$
are conjugate points.
\item[\textbf{2.}] 
The candle function is symmetric:  If $\bar{\gamma}(t) =
\gamma(r-t)$, then $s(\bar{\gamma},r) = s(\gamma,r).$
\end{enumerate}
The second property is not trivial to prove, but it is a folklore fact
in differential geometry \cite{Yau:isoperimetric}[Lem.~5] (and a standard
principle in optics).

Say that a manifold $M$ is $\Candle(\kappa)$ if the inequality 
$$s(\gamma,r) \ge s_\kappa(r)$$
holds for all $\gamma,r$; or $\LCD(\kappa)$, for \emph{logarithmic candle
derivative} \footnote{And not to be confused with liquid crystal displays.},
if the logarithmic condition
$$(\log s(\gamma,r))' \ge (\log s_\kappa(r))'$$
holds for all $\gamma,r$; or $\Ball(\kappa)$ if the volume inequality
$$\Vol(B(p,r)) \ge \Vol(B_\kappa(r))$$
holds for all $p$ and $r$; here $B_\kappa$ denotes a ball in the
simply connected space of constant curvature $\kappa$.  (If $\kappa
> 0$, then the first two conditions are only meaningful up to the
distance $\pi/\sqrt{\kappa}$ between conjugate points in the comparison
geometry.)  We also write $\Candle(\kappa,\ell)$, $\LCD(\kappa,\ell)$,
and $\Ball(\kappa,\ell)$ if the same conditions hold up to a distance of
$r = \ell$.

The logarithmic derivative $(\log s(\gamma,r))'$ of the candle function
has its own important geometric interpretation:  it is the mean curvature
of the geodesic sphere with radius $r$ and center $p = \gamma(0)$ at the
point $\gamma(r)$.  So it also equals $\Delta r$, where $\Delta$ is the
Laplace Beltrami operator, and $r$ is the distance from any point to $p$.
So if $M$ is $\LCD(\kappa)$, then we obtain the comparison $\Delta r \ge
\Delta_\kappa r_\kappa$, and the statement that spheres in $M$ are more
extrinsically curved than spheres in a space of constant curvature $\kappa$.

\subsection{Curvature and volume comparisons}

If $\kappa \le \rho = 0$, then we can organize the comparison
properties of an $n$-manifold $M$ that we have mentioned as follows:
\begin{multline}
K \le \kappa \implies \text{$\RRic$ class $(0,\kappa)$} \implies
    \LCD(\kappa) \\ \implies \Candle(\kappa) \implies \Ball(\kappa,\inj(M)),
\label{e:implies}
\end{multline}
where $\inj(M)$ is the injectivity radius of $M$.  The first implication is
elementary, while the second one is \thm{th:main}.  The third and fourth
implications are also elementary, given by integrating with respect to
length $r$.

If $\kappa \le \rho > 0$, then
\begin{multline*}
K \le \kappa \implies \text{$\RRic$ class $(\rho,\kappa)$} \implies
    \LCD(\kappa,\frac{\pi}{2\sqrt{\rho}}) \\
    \implies \Candle(\kappa,\frac{\pi}{2\sqrt{\rho}})
    \implies \Ball(\kappa,\ell),
\end{multline*}
where
$$\ell=\min(\inj(M),\frac{\pi}{2\sqrt{\rho}}).$$

Finally, for all $\ell>0$,
$$\Candle(\kappa,\ell) \implies \Ric \le (n-1)\kappa g,$$
where $g$ is the metric on $M$, because
\eq{e:deriv}{s(\gamma,r) = r^{n-1} - \Ric(\gamma'(0))r^n + O(r^{n+1}).}
In particular, in two dimensions, all of the implications in
\eqref{e:implies} are equivalences.

\subsection{Curvature bounds}

The function $\RRic(\rho)$ increases with $\rho$ faster than
$$(n-1)\sqrt{\rho-\kappa}$$
in the sense that for all $\kappa\le \rho \le \rho'$,
$$\text{$\RRic$ class $(\rho,\kappa)$} \implies \text{$\RRic$ class $(\rho',\kappa)$}.$$
In addition,
the conjugate version of root-Ricci curvature converges to normalized
Ricci curvature for large $\rho$:
$$\lim_{\rho \to \infty} \rho-\left(\frac{\RRic(\rho,u)}{n-1}\right)^2
    = \frac{\Ric(u,u)}{n-1} \quad \forall u\in UTM.$$
The corresponding limit $\rho \to \infty$ in \thm{th:main} has the
interpretation that the upper bound looks more and more like a bound
based on Ricci curvature at short distances.  This is an optimal limit
in the sense that Ricci curvature is the first non-trivial derivative of
$s(\gamma,r)$ at $r=0$ by \eqref{e:deriv}.  On the other hand, without the
length restriction, the limit $\rho \to \infty$ is impossible.  That limit
would be exactly G\"unther's inequality with Ricci curvature, but such an
inequality is not generally true.

Finally we can deduce a root-Ricci upper bound from a combination of
sectional curvature and Ricci bounds.  The concavity of the square root
function implies that given the value of $\Ric(u,u)$, the weakest possible
value of $\RRic(\rho,u)$ is achieved when $R(\cdot,u,\cdot,u)$ has one
small eigenvalue and all other eigenvalues equal.  For all $\kappa \le
\alpha \le \rho$, we then get a number $\beta=\beta(\kappa,\alpha,\rho)$,
decreasing in $\alpha$, such that
\eq{e:mixed}{K \le \alpha \text{ and } \Ric \le \beta g
    \implies \text{$\RRic$ class $(\rho,\kappa)$}.}
An explicit computation yields the optimal value
$$\beta = \rho+(n-2)\alpha - \left((n-1)\sqrt{\rho-\kappa}
    - (n-2)\sqrt{\rho-\alpha}\right)^2.$$
In particular, 
\begin{align*}
\beta(\kappa,\rho,\rho) &= (n-1)^2\kappa-n(n-1)\rho\\
\beta(\kappa,\kappa,\rho) &= (n-1)\kappa.
\end{align*}
In order to deduce $\RRic$ class $(\rho,\kappa)$ from classical curvature
upper bounds, we can therefore ask for the strong condition $K \le \kappa$
(which implies $\Ric \le (n-1)\kappa g$), or ask for the weaker $K \le
\rho$ together with $\Ric \le \beta(\kappa,\rho,\rho) g$, or choose from
a continuum of combined bounds on $K$ and $\Ric$.  Moreover, the above
calculation holds pointwise, so that in \eqref{e:mixed}, $\alpha$ can be
a function on $UTM$ instead of a constant.

\section{Applications}
\label{s:applications}

Most of the established applications of G\"unther's inequality are also
applications of \thm{th:main}.  The subtlety is that different
applications use different criteria in the chain of implications
\eqref{e:implies}.  We give some examples.  In general, let $\tM$ denote
the universal cover of $M$.

\subsection{Exponential growth of balls}
\label{s:exp}

One evident application of our result is to estimate the rate of growth of
balls, as already given by \eqref{e:implies}.  This is related
to the volume entropy of a closed Riemannian manifold $M$,
which is by definition
$$\hvol(M) \defeq  \lim_{r \to +\infty} \frac{\log \Vol B_{\tM}(p,r)}{r}.$$
By abuse of notation, we will use this same volume entropy 
expression when $M = \tM$ is simply connected rather than closed.
Since a hyperbolic space of curvature $\kappa<0$ and dimension $n$
has volume entropy $(n-1)\sqrt{-\kappa}$, \thm{th:main}
implies that when $K \le 0$,
\eq{e:vol}{\hvol(M) \ge \alpha \defeq \inf_u \RRic(0,u).}
The estimate \eqref{e:vol} is sharp for every rank one symmetric
space. (Recall that the rank one symmetric spaces are the
generalized hyperbolic spaces $\RH^n$, $\CH^n$, $\HH^n$, and
$\OH^2$.)  The reason is that the operator
$R(\cdot,\gamma',\cdot,\gamma')$ is constant along any geodesic $\gamma$.
So by the Jacobi field equation (\sec{s:proof}), the volume of $B(p,r)$
has factors of $\sinh \sqrt{\lambda_k} r$ for each eigenvalue $\lambda_k$
of $R(\cdot,\gamma',\cdot,\gamma')$.  So we obtain the estimate
$$\Vol B(p,r) \propto \prod_k (\sinh \sqrt{\lambda_k} r)
    \sim \exp(\alpha r).$$
    
However, although \eqref{e:vol} is a good estimate, it is superseded
by the previous discovery of $\RRic(0,u)$, for the specific purpose
of estimating entropies.  In addition to the volume entropy of $M$,
the geodesic flow on $M$ has a topological entropy $\htop(M)$ and a
measure-theoretic entropy $\hmu(M)$ with respect to any invariant measure
$\mu$.  Manning \cite{Manning:entropy} showed that $\htop(M) \ge \hvol(M)$
for any closed $M$, with equality when $M$ is nonpositively curved.
Goodwyn \cite{Goodwyn:compare} showed that $\htop(M) \ge \hmu(M)$ for any
$\mu$, with equality for the optimal choice of $\mu$.  (In fact he showed
this for any dynamical system.)

With these background facts, Osserman and Sarnak \cite{OS:entropy} defined
$\RRic(0,u)$ and established that
\eq{e:OSBW}{\hmu(M) \ge \int_{UTM} \RRic(0,u) d\mu(u)}
when $M$ is negatively curved, \ie, $K \le \kappa < 0$, and $\mu$ is
normalized Riemannian measure on $UTM$.  This result was generalized to
non-positive curvature by Ballmann and Wojtkowski \cite{BW:entropy}.

This use of $\RRic$ curvature concludes a topic that began with the
Schwarz-Milnor theorem \cite{Milnor:curvature,Schwarz:volume} that if $M$
is negatively curved, then $\pi_1(M)$ has exponential growth.  Part of their
result is that if $M$ is compact, then $\pi_1(M)$ has exponential growth
if and only if $\hvol(M) > 0$.  So equation \eqref{e:OSBW}, together with
Manning's theorem, shows that if $M$ is compact and nonpositively curved,
then either $M$ is flat, or the growth of $\pi_1(M)$ is bounded
below by \eqref{e:OSBW}.

Ballmann \cite{Ballmann:lectures} also showed that a non-positively curved
manifold $M$ of finite volume satisfies the weak Tits alternative: either $M$
is flat, or its fundamental group contains a non-abelian free group.  This is
qualitatively a much stronger version of the Schwarz-Milnor theorem, and
even its extension due to Manning, Osserman, Sarnak, Ballmann, and Wojtkowski.

\subsection{Isoperimetric inequalities}

Yau \cite{Yau:isoperimetric,BZ:inequalities} established that if $M$ is
complete, simply connected, and has $K \le \kappa < 0$, and $D \subseteq M$
is a domain, then $D$ satisfies a linear isoperimetric inequality:
$$\Vol(\partial D) \ge (n-1)\sqrt{-\kappa}\Vol(D).$$
His proof only uses a weakening of condition $\LCD(\kappa)$, namely that
$$(\log s(\gamma,r))' \ge (n-1)\sqrt{-\kappa}.$$
So \thm{th:main} yields Yau's inequality when $M$ is of $\RRic$ class
$(0,\kappa)$.

McKean \cite{McKean:spectrum} showed that the same weak $\LCD(\kappa)$
condition also implies a spectral gap
$$\lambda_0(\tilde M) \ge \frac{-\kappa n^2}{4}$$
for the first eigenvalue of the positive Laplace-Beltrami operator acting
on $L^2(M)$. This spectral gap follows from a Poincar\'e inequality
that is independently interesting:
$$\int_M f^2 \le \frac{4}{-\kappa n^2} \int_M |\nabla f|^2$$
for all smooth, compactly supported functions $f$.  McKean stated his
result under the hypothesis $K \le \kappa$; it has been generalized by Setti
\cite{Setti:ricci} and Borb\'ely \cite{Borbely:spectrum} to mixed sectional
and Ricci bounds; \thm{th:main} provides a further generalization. Note
in particular that Borb\'ely's result is optimal for complex hyperbolic
spaces (and we get the same bound in this case), but we get better bounds
for quaternionic and octonionic hyperbolic spaces.

Croke \cite{Croke:sharp} establishes the isoperimetric inequality for
a compact non-positively curved 4-manifold $M$ with unique geodesics.
In other words, if $B$ is a Euclidean 4-ball with
$$\Vol(M) = \Vol(B),$$
then
$$\Vol(\partial M) \ge \Vol(\partial B).$$
His proof only uses the condition $\Candle(0)$, in fact only for maximal
geodesics between boundary points\footnote{We credit \cite{Croke:sharp}
as our original motivation for this article.}.  So, Croke's theorem also
holds if $M$ is of $\RRic$ class $((\frac{\pi}{2 L})^2,0)$,
where $L$ is the maximal length of a geodesic; for any given $L$, this
curvature bound is weaker than $K \le 0$.  It is a well-known conjecture
that if $M$ is $n$-dimensional and non-positively curved, then the
isoperimetric inequality holds.  The conjecture can be attributed to Weil
\cite{Weil:negative}, because his proof in dimension $n=2$ initiated the
subject.  More recently, Kleiner \cite{Kleiner:isoperimetric} established
the case $n=3$.  We are led to ask whether Weil's isoperimetric conjecture
still holds for $\Candle(0)$ or $\LCD(0)$ manifolds.

In a forthcoming paper, we will partially generalize Croke's result to
signed curvature bounds.  In these generalizations, the main direct
hypotheses are the $\Candle(\kappa)$ and $\LCD(\kappa)$ conditions,
which are natural but not local.  \thm{th:main} provides important local
conditions under which these hypotheses hold.

\subsection{Almost non-positive curvature}

As mentioned above, one strength of root-Ricci curvature estimates is
that we can adjust the parameter $\rho$; however, most of the applications
mentioned so far are in the non-positively curved case $\rho = 0$.  It is
therefore natural to ask to which extent manifolds with almost non-positive
sectional curvature and negative root-Ricci curvature behave like negatively
curved manifolds.

More precisely, suppose that $M$ is compact, has diameter $\delta$ and
satisfies both curvature bounds
$$K\le \rho \quad \mbox{and} \quad \RRic(\rho) \le \kappa. $$
Say that $M$ is \emph{almost non-positively curved} if $0<\rho\ll
\delta^{-2}$, and that $M$ is \emph{strongly negatively root-Ricci curved}
if $\kappa\ll -\delta^{-2}$.  Under these assumptions, \thm{th:main}
shows that the balls in $\tM$ grow exponentially up to a large multiple of
the diameter $\delta$.  We conjecture that if $M$ is also compact, then
$\pi_1(M)$ has exponential growth or equivalently that $M$ has positive
volume entropy.

In light of Ballmann's result that a non-positively curved manifold $M$
of finite volume satisfies the weak Tits alternative, we ask whether a
compact, almost-non-positively curved, strongly negatively root-Ricci curved
manifold must contain a non-abelian free group in its fundamental group.
We conjecture at the very least that an almost non-positively curved manifold
with strongly negative root-Ricci curvature cannot be a torus. This would
be an interesting complement to the result of Lohkamp \cite{Lohkamp:metrics}
that every closed manifold of dimension $n \ge 3$ has a Ricci-negative
metric.

\section{The proof}
\label{s:proof}

In this section, we will prove \thm{th:main}.  The basic idea is to analyze
the energy functional that arises in a standard proof of G\"unther's
inequality, with the aid of the change of variables $R = A^2 - \rho I$.

Using the Jacobi field model, \thm{th:main} is really a result about
linear ordinary differential equations.  The normal bundle to the geodesic
$\gamma(t)$ can be identified with $\R^{n-1}$ using parallel transport.
Then an orthogonal vector field $y(t)$ along $\gamma$ is a Jacobi field
if it satisfies the differential equation
\eq{e:vector}{y'' = -R(t)y,}
where
$$R(t) = R(\cdot,u(t),\cdot,u(t))$$
is the sectional curvature matrix and $u(t) = \gamma'(t)$ is the unit
tangent to $\gamma$ at time $t$.  By the first Bianchi identity, $R(t)$ is
a symmetric matrix.  The candle function $s(r) = s(\gamma,r)$ is determined
by a matrix solution
\eq{e:matrix}{Y'' = -R(t)Y \qquad Y(0) = 0}
by the formula
$$s(r) = \frac{\det Y(r)}{\det Y'(0)}.$$
Its logarithmic derivative is given by
$$(\log s(r))' = \frac{s'(r)}{s(r)}
    = \frac{(\det Y)'(r)}{\det Y(r)}.$$
All invertible solutions $Y(r)$ to \eqref{e:matrix} are equivalent
by right multiplication by a constant matrix, and yield the same value for
$s(r)$ and its derivative.  In particular, if we let $Y(r) = I$, then the
logarithmic derivative simplifies to
$$(\log s(r))' = \Tr(Y'(r)).$$

Following a standard proof of G\"unther's inequality
\cite{GHL:riemannian}[Thm.~3.101], we define an energy functional whose
minimum, remarkably, both enforces \eqref{e:matrix} and minimizes the
objective $(\log s(r))'$. Namely, we assume Dirichlet boundary conditions
$$y(0) = 0 \qquad y(r) = v,$$
and we let
\eq{e:energy}{E(R,y) = \int_0^r \left(\braket{y',y'}
    - \braket{y,Ry}\right) \,\dd t.}
By a standard argument from calculus of variations, the critical points
of $E(R,y)$ are exactly the solutions to \eqref{e:vector} with the given
boundary conditions.

We can repeat the same calculation with the matrix solution 
$$Y(0) = 0 \qquad Y(r) = I,$$
with the analogous energy
$$E(R,Y) = \int_0^r \left(\braket{Y',Y'} - \braket{Y,RY}\right) \,\dd t.$$
Here the inner product of two matrices is the Hilbert-Schmidt inner product
$$\braket{A,B} = \Tr(A^TB).$$
Moreover, if $Y$ is a solution to \eqref{e:matrix}, then $E(R,Y)$ simplifies to
$(\log s(r))'$ by integration by parts:
\begin{align*}
E(R,Y) &= \int_0^r \left(\braket{Y',Y'} - \braket{Y,RY}\right) \,\dd t \\
    &= \braket{Y(r),Y'(r)} - \braket{Y(0),Y'(0)}
    - \int_0^r \braket{Y,Y''+RY} \,\dd t \\
    &= \braket{I,Y'(r)} - 0 - 0 = \Tr(Y'(r)) = (\log s(r))'.
\end{align*}
Thus, our goal is to minimize $E(R,Y)$ with respect to both $Y$ and $R$.
We want to minimize with respect to $Y$ in order to solve \eqref{e:matrix}.
Then for that $Y$, we want to minimize with respect to $R$ to prove
\thm{th:main}.

The following proposition tells us that \eqref{e:vector} or \eqref{e:matrix}
has a unique solution with Dirichlet boundary conditions, and that it is
an energy minimum.  Here and below, recall the matrix notation $A \le B$
(which was already used for Ricci curvature in the introduction) to express
the statement that $B-A$ is positive semidefinite.

\begin{proposition} If $R \le \rho I$, and if $y$ is continuous with an $L^2$
derivative, then $E(R,y)$ is a positive definite quadratic function of $y$
when $\sqrt{\rho}r < \pi$, with the Dirichlet boundary conditions $y(0)
= y(r) = 0$.
\label{p:posdef} \end{proposition}

\begin{proof} 
Let
$$E(\rho,y) = E(\rho I,y) = \int_0^r \left(\braket{y',y'}
    - \rho\braket{y,y}\right)\,\dd t$$
be the corresponding energy of the comparison case with constant curvature
$\rho$.  (Recall that the ultimate comparison is with constant curvature
$\kappa$, but to get started we use $\rho$ instead.)  Then
$$E(\rho,y) \le E(R,y),$$
so it suffices to show that $E(\rho,y)$ is positive definite.  When $\rho
= 0$, $E(\rho,y)$ is manifestly positive definite.  Otherwise $E(\rho,y)$
is diagonalized in the basis of functions
$$y_k(t) = \sin(\frac{\pi k t}{r})$$
with $k \ge 1$. A direct calculation yields
$$E(\rho,y_k) = \frac{\pi^2 k^2-r^2\rho}{r} > 0,$$
as desired.
\end{proof}

\begin{remark} There is also a geometric reason that the comparison case
$E(\rho,y)$ is positive definite:  When $\rho = 0$, a straight line segment
in Euclidean space is a minimizing geodesic; when $\rho > 0$, the same is
true of a geodesic arc of length $r < \pi/\sqrt{\rho}$ on a sphere with
curvature $\sqrt{\rho}$.  We give a direct calculation to stay in the
spirit of ODEs.
\end{remark}

\begin{proposition} Let $\rho$ and $r < \pi/\sqrt{\rho}$ be fixed and suppose
that $R \le \rho I$.  Then $s(r)$ and $(\log s(r))'$ are both bounded below.
\label{p:bound1} \end{proposition}

\begin{proof} We will simply prove the usual G\"unther inequality.
As in the proof of \prop{p:posdef},
$$E(R,Y) \ge E(\rho,Y)$$
for all $R$ and $Y$ with $Y(0) = 0$ and $Y(r) = I$.  For each fixed $R$,
the minimum of the left side is $(\log s(r))'$.  The minimum of the right
side (which may occur for a different $Y$, but no matter) is
$(\log s_{\rho}(r))'$, which is a positive number.  We obtain
the same conclusion for $s(r)$ by integration.
\end{proof}

\begin{proposition} Assume the hypotheses of \prop{p:bound1}.  If $R$ is
$L^\infty$, then the solution $Y$ to \eqref{e:matrix} is bounded uniformly,
\ie, with a bound that depends only on $||R||$ (and $r$ and $\rho$).
Also $Y'$ is uniformly bounded and Lipschitz, and $Y''$ is uniformly
bounded and $L^\infty$.
\label{p:bound2} \end{proposition}

\begin{proof} In this proposition and nowhere else, it is more convenient
to assume the initial conditions
$$\hY(0) = 0 \qquad \hY'(0) = I$$
rather than Dirichlet boundary conditions.  The fact that $\hY$ and its
derivatives are uniformly bounded, with these initial conditions, is exactly
Gr\"onwall's inequality.  To convert back to Dirichlet boundary conditions,
we want to instead bound
$$Y(t) = \hY(t)\hY(r)^{-1}.$$
This follows from \prop{p:bound1} by the formula
$$\hY(r)^{-1} = \adj(\hY(r))\det(\hY(r))^{-1},$$
where $\adj$ denotes the adjugate of a matrix.

Finally, $Y''(t)$ is $L^\infty$ and uniformly bounded because $Y(t)$
satisfies \eqref{e:matrix}.  Also $Y'(0) = \hY(r)^{-1}$ is uniformly bounded,
so we can integrate to conclude that $Y'(t)$ is uniformly bounded and
Lipschitz.
\end{proof}

To prove \thm{th:main}, we want to minimize $(\log s(r))'$ or $E(R,Y)$
over all $R$ such that
\eq{e:Rcond}{R \le \rho I \qquad \Tr(\sqrt{\rho I-R})
    \ge \alpha \defeq \sqrt{\rho-\kappa}.}
To better understand this minimization problem, we make a change of
variables.   Let $A(t)$ be a symmetric matrix such that
\eq{e:A}{R(t) = \rho I - A(t)^2  \qquad \Tr(A(t)) \ge \alpha.}
In order to know that every $R(t)$ is realized, we can let
$$A = \sqrt{\rho I-R}$$
be the positive square root of $\rho I-R$.  Even if $A$ is not positive
semidefinite, $R(t)$ still satisfies \eqref{e:Rcond}.  This simplifies the
optimization problem:  in the new variable $A$, the semidefinite hypothesis
can be waived.

Now the energy function becomes:
\begin{align*}
E(A,Y) &= \int_0^r \left(\braket{Y',Y'}
    - \braket{Y,(\rho-A^2)Y}\right)\,\dd t \\
    &= \int_0^r \left( \Tr((Y')^TY') + \Tr(Y^TA^2Y)
    - \rho\Tr(Y^TY)\right)\,\dd t.
\end{align*}
For the moment, fix $Y$ and let $Z = YY^T$.  Then as a function of $A$,
$$E(A) = \int_0^r \Tr(A^2YY^T)\, \dd t + \text{constant}.$$
Since $YY^T$ is symmetric and strictly positive definite, $E$ is a
positive-definite quadratic function of $A$, and we can directly solve
for the minimum as
\eq{e:AYY}{A = \frac{\alpha (YY^T)^{-1}}{\Tr((YY^T)^{-1})}.}
Even though we waived the assumption that $A$ is positive semidefinite,
minimization restores it as a conclusion.  Moreover,
\eq{e:trap}{\Tr(A) = \Tr(\sqrt{\rho I-R}) = \alpha.}

\begin{proposition} With the hypotheses \eqref{e:Rcond}, and if $r <
\pi/\sqrt{\rho}$, a minimum of $(\log s(r))'$ exists. Equivalently,
a joint minimum of $E(A,Y)$ or $E(R,Y)$ exists.
\label{p:exists} \end{proposition}

\begin{proof} The above calculation lets us assume \eqref{e:trap},
which means that $R$ is uniformly bounded.  By \prop{p:bound2}, so is
$Y''$.  We can restrict to a set of pairs $(R,Y'')$ of class $L^\infty$,
which is compact in the weak-* topology by the Banach-Alaoglou theorem.
Equivalently, we can restrict to a uniformly bounded, uniformly Lipschitz
set of pairs $(\smallint R,Y')$, which is compact in the uniform topology by
the Arzela-Ascoli theorem.  By integration by parts, we can write
\begin{align*}
E(R,Y) &= \int_0^r \left(\braket{Y',Y'} - \braket{Y,RY}\right) \,\dd t. \\
    &= \left[\braket{Y,(\smallint R) Y} \right]_0^r +
    \int_0^r \left(\braket{Y',Y'} \,\dd t + 2\braket{Y',(\smallint R) Y} \right) \,\dd t.
\end{align*}
Thus the energy is continuous as a function of $\smallint R$ and $Y'$ and
has a minimum on a compact family.
\end{proof}

\prop{p:exists} reduces \thm{th:main} to solving the following
non-linear matrix ODE, which is obtained by combining \eqref{e:matrix}
and \eqref{e:AYY}:
\begin{align*}
Y'' &= (A^2-\rho)Y & A &= \frac{\alpha (YY^T)^{-1}}{\Tr((YY^T)^{-1})} \\
Y(0) &= 0 & Y(r) &= I.
\end{align*}
\prop{p:exists} tells us that this ODE has at least one solution; we will
proceed by finding all solutions with the given boundary conditions.
First, if we suppress the boundary condition $Y(r) = I$, the solutions
$Y(t)$ are invariant under both left and right multiplication by $O(n-1)$.
So we can write
$$Y(t) = U\hY(t)V,$$
where $\hY'(0)$ is diagonal with positive entries.  In this case $\hA(0)$
is also diagonal, and we obtain that $\hY(t)$ is diagonal for all $t$, and
with positive entries because the entries cannot cross 0.  Therefore $UV
= I$, because the identity is the only diagonal orthogonal matrix with
positive entries.

So we can assume that $Y = \hY$, with diagonal entries
$$\lambda_1(t), \lambda_2(t), \ldots, \lambda_{n-1}(t) > 0.$$
Each of these entries satisfies the same scalar ODE,
\eq{e:scalar}{w'' = \beta(t) w^{-1} - \rho w \qquad w(0) = 0 \qquad w(r) = 1,}
where
$$\beta(t) = \frac{\alpha}{\Tr((Y(t)Y(t)^T)^{-1})^2}.$$
We claim that if $w > 0$, then $w' > 0$ as well.  If $\rho = 0$,
then this is immediate.  Otherwise, a positive solution $w(t)$ satisfies
$$w(t) > \frac{\sin(\sqrt{\rho}t)}{\sin(\sqrt{\rho}r)} \qquad
w'(t) > \frac{\sqrt{\rho}\cos(\sqrt{\rho}t)}{\sin(\sqrt{\rho}r)},$$
because the right side is the solution to $w'' = - \rho w$ with the same
boundary conditions.  So we obtain that $w' > 0$ provided that
$$r < \frac{\pi}{2\sqrt{\rho}}.$$
(This is where we need half of the distance allowed in the usual form of
G\"unther's inequality.)

To complete the proof, consider the phase diagram in the strip $[0,1] \times
(0,\infty)$ of the positive solutions $(w(t),w'(t))$ to \eqref{e:scalar}.
If we let $x = w(t)$, then the total elapsed time to reach $x = 1$ is
$$r = \int_0^1 \frac{\dd t}{\dd x}\dd x
    = \int_0^1 \frac{\dd x}{w'(w^{-1}(x))},$$
which is a positive integral.  On the other hand, if $w_1$ and $w_2$
are two distinct solutions with
$$w_1(0) = w_2(0) = 0 \qquad w'_1(0) > w'_2(0),$$
then the solutions cannot intersect in the phase diagram; we must have
$$w'_1(w_1^{-1}(x)) > w'_2(w_2^{-1}(x)) > 0.$$
So two distinct, positive solutions to \eqref{e:scalar} cannot reach $w(t)
= 1$ at the same time, which means with given the boundary conditions that
there is only one solution.  Thus, the diagonal entries $\lambda_k(t)$ of
$Y(t)$ are all equal.  In conclusion, $Y$, $A$, and $R$ all are isotropic
at the minimum of the logarithmic candle derivative $(\log s(r))'$.
This additional property implies the estimate for $(\log s(r))'$ immediately.
(Note that when $R$ is isotropic, the hypothesis becomes equivalent to
$K\le \kappa$, the usual assumption of G\"unther's inequality.)

\bibliography{dg}

\providecommand{\bysame}{\leavevmode\hbox to3em{\hrulefill}\thinspace}
\providecommand{\MR}{\relax\ifhmode\unskip\space\fi MR }
\providecommand{\MRhref}[2]{%
  \href{http://www.ams.org/mathscinet-getitem?mr=#1}{#2}
}
\providecommand{\href}[2]{#2}
\providecommand{\eprint}{\begingroup \urlstyle{tt}\Url}
\begin{thebibliography}{10}

\bibitem{Ballmann:lectures}
Werner Ballmann, \emph{Lectures on spaces of nonpositive curvature}, DMV
  Seminar, vol.~25, Birkh\"auser Verlag, 1995, With an appendix by Misha Brin.

\bibitem{BW:entropy}
Werner Ballmann and Maciej~P. Wojtkowski, \emph{An estimate for the
  measure-theoretic entropy of geodesic flows}, Ergodic Theory Dynam. Systems
  \textbf{9} (1989), no.~2, 271--279.

\bibitem{BC:book}
Richard~L. Bishop and Richard~J. Crittenden, \emph{Geometry of manifolds}, Pure
  and Applied Mathematics, vol.~XV, Academic Press, 1964.

\bibitem{Borbely:spectrum}
Albert Borb\'ely, \emph{On the spectrum of the {Laplacian} in negatively curved
  manifolds}, Studia Sci. Math. Hungar. \textbf{30} (1995), no.~3-4, 375--378.

\bibitem{BZ:inequalities}
Yuri Burago and Victor~A. Zalgaller, \emph{Geometric inequalities}, Grundlehren
  der Mathematischen Wissenschaften [Fundamental Principles of Mathematical
  Sciences], vol. 285, Springer-Verlag, 1988, Translated from the Russian by A.
  B. Sosinski{\u\i}, Springer Series in Soviet Mathematics.

\bibitem{Croke:sharp}
Christopher~B. Croke, \emph{A sharp four-dimensional isoperimetric inequality},
  Comment. Math. Helv. \textbf{59} (1984), no.~2, 187--192.

\bibitem{GHL:riemannian}
Sylvestre Gallot, Dominique Hulin, and Jacques Lafontaine, \emph{Riemannian
  geometry}, second ed., Universitext, Springer-Verlag, 1990.

\bibitem{Goodwyn:compare}
L.~Wayne Goodwyn, \emph{Comparing topological entropy with measure-theoretic
  entropy}, Amer. J. Math. \textbf{94} (1972), 366--388.

\bibitem{Gunther:volume}
Paul G\"unther, \emph{Einige {S\"atze} \"uber das {Volumenelement} eines
  {Riemannschen Raumes}}, Publ. Math. Debrecen \textbf{7} (1960), 78--93.

\bibitem{Kleiner:isoperimetric}
Bruce Kleiner, \emph{An isoperimetric comparison theorem}, Invent. Math.
  \textbf{108} (1992), no.~1, 37--47.

\bibitem{Lohkamp:metrics}
Joachim Lohkamp, \emph{Metrics of negative {Ricci} curvature}, Ann. of Math.
  (2) \textbf{140} (1994), no.~3, 655--683.

\bibitem{Manning:entropy}
Anthony Manning, \emph{Topological entropy for geodesic flows}, Ann. of Math.
  (2) \textbf{110} (1979), no.~3, 567--573.

\bibitem{McKean:spectrum}
Henry~P. McKean, \emph{An upper bound to the spectrum of {$\Delta$} on a
  manifold of negative curvature}, J. Differential Geometry \textbf{4} (1970),
  359--366.

\bibitem{Milnor:curvature}
John Milnor, \emph{A note on curvature and fundamental group}, J. Differential
  Geometry \textbf{2} (1968), 1--7.

\bibitem{OS:entropy}
Robert Osserman and Peter Sarnak, \emph{A new curvature invariant and entropy
  of geodesic flows}, Invent. Math. \textbf{77} (1984), no.~3, 455--462.

\bibitem{Schwarz:volume}
Albert Schwarz, \emph{A volume invariant of coverings}, Dokl. Akad. Nauk SSSR
  (N.S.) \textbf{105} (1955), 32--34.

\bibitem{Setti:ricci}
Alberto~G. Setti, \emph{A lower bound for the spectrum of the {Laplacian} in
  terms of sectional and {Ricci} curvature}, Proc. Amer. Math. Soc.
  \textbf{112} (1991), no.~1, 277--282.

\bibitem{Weil:negative}
Andr\'e Weil, \emph{Sur les surfaces a courbure negative}, C. R. Acad. Sci.
  Paris \textbf{182} (1926), 1069--1071.

\bibitem{Yau:isoperimetric}
Shing-Tung Yau, \emph{Isoperimetric constants and the first eigenvalue of a
  compact {Riemannian} manifold}, Ann. Sci. \'Ecole Norm. Sup. (4) \textbf{8}
  (1975), no.~4, 487--507.

\end{thebibliography}

\end{document}